\documentclass{article}

\usepackage{fullpage,amsthm,amsmath,amssymb,amsbsy,amsfonts}

\usepackage{graphicx}

\title{\bf Sharp estimates for the gradient of solutions to the heat equation}

\author{\sc{Gershon Kresin$^a\!\!$}
\thanks{Corresponding author. E-mail: kresin@ariel.ac.il}$\;\;$
 and \sc{Vladimir Maz'ya$^b$}
\thanks{E-mail: vladimir.mazya@liu.se}$\;\;$ 
\\ \\
{\it{$^a$Department of Mathematics, Ariel University, Ariel 40700, Israel}}\\
{\it{$^b$Department of Mathematical Sciences, University of Liverpool,
M$\&$O Building, Liverpool,}}\\ 
{\it{L69 3BX, UK; Department of Mathematics, Link\"oping University,SE-58183 Link\"oping, }}\\
{\it{{\hskip -19mm}Sweden; }}
{\it{RUDN University, 6 Miklukho-Maklay St., Moscow, 117198, Russia}}\\
}
{ \date\ }

\numberwithin{equation}{section}
\newtheorem{lemma}{Lemma}
\newtheorem{theorem}{Theorem}
\newtheorem{proposition}[theorem]{Proposition}
\newtheorem{corollary}{Corollary}
\newenvironment{remark}{{\bf Remark}}

\newcommand{\bs}{\boldsymbol} 
\newcommand{\nl}{\lVert}
\newcommand{\nr}{\rVert}

\begin{document}
\maketitle
\large
\centerline{\sl In memory of Solomon G. Mikhlin}
\vspace{10mm}

{\bf Abstract.} Various sharp pointwise estimates for the gradient of solutions to the heat equation are obtained.
The Dirichlet and Neumann conditions are prescribed on the boundary of a half-space. 
All data belong to the Lebesgue space $L^p$. Derivation of the coefficients is based on solving 
certain optimization problems with respect to a vector parameter inside of an integral over the unit sphere. 
\\
\section{Introduction}\label{CH_10I}

In the present paper we find the best coefficients in certain inequalities for solutions to the heat equation.
Previously results of similar nature for stationary problems were obtained in our works \cite{KM1}-\cite{KM4} and \cite{KM5}, where
solutions of the Laplace, Lam\'e and Stokes equations were considered.

In particular, in \cite{KM5} a representation for the sharp coefficient ${\mathcal A}_{n,p}(x)$ in the inequality
\begin{equation} \label{Eq_0.1}
\left |\nabla \left \{ \frac{u (x)}{x_n} \right \}\right |
\leq {\mathcal A}_{n,p}(x)\;\big|\!\big | u(\cdot, 0 )\big |\!\big |_p
\end{equation}
was derived, where $u$ is a harmonic function in the half-space ${\mathbb R}^n_+=\{ x=(x', x_n): x' \in {\mathbb R}^{n-1}, x_n >0 \}$, 
represented by the Poisson integral with boundary values in $L^p({\mathbb R}^{n-1})$, $||\cdot ||_p$ is the norm 
in $L^p({\mathbb R}^{n-1})$, $1 \leq p \leq \infty$. It was shown that
$$
{\mathcal A}_{n,p}(x)=\frac{A_{n,p}}{ x_{n}^{2+(n-1)/p}}, 
$$
where
$$
A_{n,p}=\frac{2 n}{\omega_n}\left \{\frac{\pi^{\frac{n-1}{2}}\Gamma\left ( \frac{3p+n-1}{2(p-1)} \right )}
{\Gamma\left ( \frac{(n +2)p}{2(p-1)} \right )} \right \}^{1-\frac{1}{p}}
$$
for $1< p< \infty$, and $A_{n,1}=2n/\omega_n$, $A_{n,\infty}=1$. Here and henceforth we denote by $\omega_n$ the area of the unit sphere
${\mathbb S}^{n-1}$ in ${\mathbb R}^n$.

Another sharp estimate for the modulus of the gradient of harmonic functions in ${\mathbb R}^{n}_+$ was obtained 
in \cite{KM2}: 
\begin{equation} \label{Eq_0.2}
|\nabla u(x)|\leq {\mathcal N}_{n,p}(x) \left |\left | \frac{\partial u}{\partial \bs\nu} 
\right |\right |_p,
\end{equation}
where $\bs \nu$ is the unit normal vector to $\partial {\mathbb R}^n_+$, $p \in [1, n]$, $x\in {\mathbb R}^n_+$. 
The best value of the coefficient in (\ref{Eq_0.2}) is given by
$$
{\mathcal N}_{n,p}(x)=\frac{N_{n,p}}{x_n^{(n-1)/p}}\;,
$$
where
$$
N_{n,p}=\frac{2^{1/p}}{\omega_n}\left \{
  \frac{2\pi ^{(n-1)/2}\Gamma \left (\frac{n+p-1}{2p-2} \right )}{\Gamma\left (\frac{np}{2p-2} \right )} 
   \right \}^{1-\frac{1}{p} }
$$
for $1<p \leq n$, and $N_{n,1}=1/\omega_n$.

The plan of the present paper is as follows. Section 2 is auxiliary. It is devoted to a certain optimization problem with respect
to vector parameter inside of the integral over the unit sphere of ${\mathbb R}^{n}$. In the next sections we study solutions
to the heat equation. The boundary value problem
$$
\frac{\partial u}{\partial t}=a^2\Delta u\;\;{\rm in}\;\;{\mathbb R}^{n}_+\times (0, +\infty ),\;\;u\big |_{t=0}=0,\;\;u\big |_{x_n=0}=f(x', t)
$$
is considered in Section 3. Here $f\in L^p\big ({\mathbb R}^{n-1}\times (0, +\infty ) \big )$, $1\leq p\leq \infty$, and the solution 
$u$ is represented by the heat double layer potential.
The norm in the space $L^p\big ({\mathbb R}^{n-1} \times (0, t ) \big )$ is defined by 
\begin{eqnarray} \label{Eq_0.3}
\nl f \nr_{p, t}=\left\{\begin{array}{lll}
\displaystyle{\left \{ \int_0^t\int _ {{\mathbb R}^{n-1} } | f(x', \tau ))|^p dx' d\tau\right \}^{1/p}}
\;&\quad\;\;\;{\rm for}\;\;1\leq  p< \infty\;, \\
        \\
\displaystyle{\mbox{ess}\;\sup \{ | f(x', \tau) |:x' \in {\mathbb R}^{n-1},\; \tau \in (0, t) \}}\;&\quad\quad{\rm for}\;\;p=\infty\;.
\end{array}\right .
\end{eqnarray}
The main result obtained in Section 3 is the inequality
$$
\left | \nabla_x \left \{ \frac{u(x, t)}{x_n} \right \}\right |\leq {\mathcal W}_p(x, t) ||f||_{p, t}
$$
with the best coefficient 
\begin{equation} \label{Eq_0.5}
{\mathcal W}_p(x,t)=\frac{c_{n,p}}{x_n^{2+\frac{n+1}{p}}}
\max_{|\bs z|=1}\left \{ \int_{{\mathbb S}^{n-1}}\omega_{\kappa , \lambda} \big ( ( \bs e_{\sigma}, \bs e_n )\big )
|( \bs e_{\sigma}, \bs e_n )|^{\frac{n+p+2}{p-1}}|( \bs e_{\sigma}, \bs z )|^{\frac{p}{p-1}} d\sigma\right \}^{\frac{p-1}{p}}\!,
\end{equation}
where $(x, t)$ is an arbitrary point in ${\mathbb R}^{n}_+\times (0, +\infty )$,
\begin{equation} \label{Eq_0.4}
\omega_{\kappa, \lambda}(u)=\int_{\kappa/u^2}^\infty \xi^\lambda e^{-\xi}d\xi\\;,
\end{equation}
and
$$
c_{n,p}=\frac{2^{\frac{1}{p}}(4a^2)^{1+\frac{1}{p}}}{\pi^{\frac{n}{2}-1}q^{\frac{n}{2}+1+\frac{1}{p}}}\;,\;\;\;\;
\kappa=\frac{q x_n^2}{4a^2 t}\;,\;\;\;\;\;\lambda=\frac{(n+4)q}{2}-2
$$
with $p^{-1}+q^{-1}=1$. 

The extremal problem in (\ref{Eq_0.5}) is solved for the case $2\leq p\leq \infty$ and the explicit formula 
$$
{\mathcal W}_p(x, t)=\frac{c_{n,p}}{x_n^{2+\frac{n+1}{p}}}\left \{ \!2\omega_{n-1}\int_0^{\pi/2} 
\left \{\int_{\frac{q x_n^2}{4a^2 t\cos^2\vartheta}}^\infty \xi^{\frac{np+4}{2(p-1)}}e^{-\xi}d\xi \right \}
 \cos ^{\frac{n+2(p+1)}{p-1}}\vartheta 
\sin^{n-2}\vartheta d\vartheta\!\right \}^{\frac{p-1}{p}}
$$
is obtained. In particular, 
$$
{\mathcal W}_\infty(x, t)=\frac{16a^2\sqrt{\pi}}{\Gamma \left (\frac{n-1}{2} \right )x_n^2}\int_0^{\pi/2} 
\left \{\int_{ \frac{x^2_n}{4a^2 t \cos^2\vartheta}}^\infty \xi^{n/2} e^{-\xi}d\xi \right \}
\cos ^{2}\vartheta \sin^{n-2}\vartheta d\vartheta\;.
$$

In Section \ref{S_4} we obtain an analog of (\ref{Eq_0.2}) for solutions of the Neumann problem 
$$
\frac{\partial u}{\partial t}=a^2\Delta u\;\;{\rm in}\;\;{\mathbb R}^n_+\times (0, +\infty),\;\;u\big |_{t=0}=0,\;\;
\frac{\partial u}{\partial x_n}\Big |_{x_n=0}=g(x', t)
$$
with $g\in L^p\big ({\mathbb R}^{n-1}\times (0, +\infty ) \big )$, represented by the heat single layer potential, $1\leq p \leq \infty$.

It is shown that for an arbitrary point $(x, t) \in {\mathbb R}^{n}_+\times (0, +\infty )$, the sharp coefficient ${\mathcal N}_p(x, t)$ in the inequality 
$$
|\nabla_x u(x, t)|\leq {\mathcal N}_p(x, t) ||g||_{p, t}
$$
is given by
\begin{equation} \label{Eq_0.6}
{\mathcal N}_p(x, t)=\frac{k_{n,p}}{x_n^{\frac{n+1}{p}}}\max_{|\bs z|=1}
\left \{ \int_{{\mathbb S}^{n-1}}\omega_{\kappa , \lambda} \big ( ( \bs e_{\sigma}, \bs e_n )\big )
|( \bs e_{\sigma}, \bs e_n )|^{\frac{n-p+2}{p-1}}|( \bs e_{\sigma}, \bs z )|^{\frac{p}{p\!-\!1}} d\sigma\right \}^{\frac{p-1}{p}}\!,
\end{equation}
where $\omega_{\kappa , \lambda}(u)$ is the same as in (\ref{Eq_0.4}), and
$$
k_{n,p}=\frac{2^{(3-p)/p} a^{2/p}}{\pi^{n/2}q^{\frac{n}{2}+\frac{1}{p}}}\;,\;\;\;\;\;\kappa=\frac{q x_n^2}{4a^2 t}\;,\;\;\;\;\;
\lambda=\frac{(n+2)q}{2}-2\;.
$$

The extremal problem in (\ref{Eq_0.6}) is solved for the case $2\leq p\leq (n+4)/2$ and the explicit formula 
$$
{\mathcal N}_p(x, t)=\frac{k_{n,p}}{x_n^{\frac{n+1}{p}}}
\left \{ 2\omega_{n-1}\!\! \int_0^{\pi/2} \!\!
\left \{\int_{\frac{q x_n^2}{4a^2 t \cos^2\vartheta}}^\infty \xi^{\frac{(n-2)p+4}{2(p-1)}}e^{-\xi}d\xi \right \}
\cos ^{\frac{n+2}{p-1}}\vartheta 
\sin^{n-2}\vartheta d\vartheta\right \}^{\frac{p\!-\!1}{p}}
$$
is obtained. In particular,
$$
{\mathcal N}_2(x, t)=\frac{b_{n}}{x_n^{\frac{n+1}{2}}}\left \{\int_0^{\pi/2} 
\left \{\int_{ \frac{x^2_n}{2a^2 t \cos^2\vartheta}}^\infty \xi^{n} e^{-\xi}d\xi \right \}
\cos ^{n+2}\vartheta \sin^{n-2}\vartheta d\vartheta \right \}^{1/2}\;,
$$
where
$$
b_{n}=\frac{a}{2^{\frac{n-1}{2}}\pi^{\frac{n+1}{4}}\sqrt{\Gamma \left (\frac{n-1}{2} \right )}}\;.
$$

\section{Extremal problems for integrals with parameters} \label{S_2}

\subsection{Extremal problem for integrals with parameter on the space with measure}\label{abstr_param}

Let $X$ is the space with $\sigma$-finite measure $\mu$ defined on the
$\sigma$-algebra ${\mathfrak S}$ of measurable sets, parameters $y$ and $y_{_0}$ are elements of a set $Y$,
$\rho (x; y)$ and $f(x; y)$ are $[0, +\infty ]$-valued ${\mathfrak S}$-measurable functions on $X$ 
for any fixed $y\in Y$. 

A particular case of the assertion below with $\rho\equiv 1$ and somewhat weaker assumption was proved in \cite{K}.

\begin{proposition} \label{P_1} Let $y_{_0}$ be a fixed point of $Y$. 
Let $\gamma \in (0, +\infty)$ and let the integral
\begin{equation} \label{Eq_1.1}  
\int_X \rho (x; y_{_0})f^\gamma(x; y) d\mu
\end{equation}
attains its supremum on $y\in Y$ at the point $y_{_0}\in Y$ (the case of $+\infty$ is not excluded). Further on, let
\begin{equation} \label{Eq_1.2} 
{\mathcal I}(y, y_{_0})=\int_X \rho (x; y_{_0})f^{\alpha }(x; y)f^{\beta}(x; y_{_0})d\mu\;,
\end{equation}
where $\alpha >0, \beta \geq 0$. 

Then the equality holds
\begin{equation} \label{Eq_1.3} 
\sup_{y\in Y}{\mathcal I}(y, y_{_0})={\mathcal I}(y_{_0}, y_{_0})=
\int_X \rho (x; y_{_0})f^\gamma (x; y_{_0})d\mu
\end{equation}
for any $\alpha$ and $\beta $ such that $\alpha+\beta=\gamma$.

In particular, the supremum of ${\mathcal I}(y, y_{_0})$ over $y\in Y$ is independent of $y_{_0}$ if the value 
of integral 
$$
\int_X \rho (x; y)f^\gamma (x; y) d\mu
$$
does not depend on $y$.
\end{proposition}  
\begin{proof} Let $\alpha >0$ and $\beta \geq 0$ are arbitrary numbers, $\alpha+\beta=\gamma$. 
The case $\beta=0$ is obvious. Now, let $\beta >0$. By H\"older's inequality, the integral
\begin{eqnarray*} 
{\mathcal I}(y, y_{_0})&=&\int_X\! \rho (x; y_{_0})f^{\alpha}(x; y)f^{\beta}(x; y_{_0})d\mu\\
& &\\
&=&\int_X\! \left (\rho^{\frac{\alpha}{\gamma}}(x; y_{_0})f^{\alpha}(x; y)\right )
\left (\rho^{\frac{\beta}{\gamma}}(x; y_{_0})f^{\beta}
(x; y_{_0})\right )d\mu
\end{eqnarray*}
does not exceed the product
$$
\left \{\int_X \rho^{\frac{\alpha}{\gamma}\frac{\gamma}{\alpha}}(x; y_{_0})f^{{\alpha }
\frac{\gamma}{\alpha}}(x; y) d\mu \right \}^{\frac{\alpha }{\gamma}}
\left \{\int_X \rho^{\frac{\beta}{\gamma}\frac{\gamma}{\beta}}(x; y_{_0})
f^{{\beta}\frac{\gamma}{\beta}}(x; y_{_0}) d\mu \right \}^{\frac{\beta}{\gamma}}.
$$
Since integral (\ref{Eq_1.1}) attains its supremum on $y\in Y$ at $y_{_0}$, it follows that
\begin{equation} \label{Eq_1.4}
\sup_{y\in Y}{\mathcal I}(y, y_{_0})\leq \int_X \rho(x; y_{_0})f^{\gamma} (x; y_{_0})d\mu\;.
\end{equation}

On the other hand, by (\ref{Eq_1.2}) we have
$$
\sup_{y\in Y}{\mathcal I}(y, y_{_0})\geq {\mathcal I}(y_{_0}, y_{_0})=
\int_X\! \rho(x; y_{_0})f^{\gamma}(x; y_{_0})d\mu \;,
$$
which together with (\ref{Eq_1.4}) completes the proof.
\end{proof}

\subsection{Extremal problem for integral over ${\mathbb S}^{n-1}$}\label{unit_sphere}

Let $\bs e_\sigma$ be the $n$-dimensional unit vector joining the origin to a point $\sigma \in {\mathbb S}^{n-1}$. 
We denote by $\bs e$ and $\bs z$ the $n$-dimensional unit vectors and assume that $\bs e$ is a fixed vector.
Let $\rho $ and $f$ be non-negative Lebesgue measurable functions in $[-1, 1]$. 

The next assertion is an immediate consequence of Proposition \ref{P_1}.
\begin{corollary} \label{C_1} Let $\gamma >0$ and let the integral
\begin{equation} \label{Eq_1.5}
\int_{{\mathbb S}^{n-1}} \rho\big ((\bs e_\sigma, \bs e) \big )f^{\gamma}\big ( (\bs e_\sigma, \bs z) \big )d\sigma
\end{equation}
attains its supremum on $\bs z\in {\mathbb R}^n$, $|\bs z|=1$ at the vector $\bs e$. Further, let $\alpha \geq 0, \beta > 0$ and $\alpha+\beta=\gamma$. 
Then
\begin{align} \label{Eq_1.6}
&\sup _{|\bs z|=1}\int_{{\mathbb S}^{n-1}}\rho\big ((\bs e_\sigma, \bs e) \big )f^{\alpha }\big ((\bs e_\sigma, \bs e)\big )f^{\beta }\big ((\bs e_\sigma, \bs z)\big )d\sigma\nonumber\\
&=\int_{{\mathbb S}^{n-1}}\rho\big ((\bs e_\sigma, \bs e) \big )f^{\gamma}
\big ((\bs e_\sigma, \bs e) \big )d\sigma .
\end{align}
\end{corollary}

\begin{remark} By the equality
$$
\int_{{\mathbb S}^{n-1}} F\big ((\bs e_\sigma, \bs e) \big )d\sigma
=\omega_{n-1}\int_0^\pi F\big ( \cos \vartheta \big )\sin ^{n-2}\vartheta d\vartheta , 
$$
we conclude that value of the integral in the right-hand side of (\ref{Eq_1.6}) is independent of $\bs e$. 
In the case of the even function $F$, the last equality can be written as 
\begin{equation} \label{Eq_1.6A}
\int_{{\mathbb S}^{n-1}} F\big ((\bs e_\sigma, \bs e) \big )d\sigma
=2\omega_{n-1}\int_0^{\pi/2} F\big ( \cos \vartheta \big )\sin ^{n-2}\vartheta d\vartheta . 
\end{equation}
\end{remark}

\bigskip
Further, we consider a special case of Corollary \ref{C_1} with $\gamma=2$, 
\begin{equation} \label{Eq_1.6R}
\rho_{\kappa, \lambda, \mu}(u)=\omega_{\kappa, \lambda}(u)|u|^\mu,\;\;\;\;\;\;f(u)=|u|,
\end{equation}
where $\kappa, \lambda , \mu \geq 0$ and 
\begin{equation} \label{Eq_1.6AB}
\omega_{\kappa, \lambda}(u)=\int_{\kappa/u^2}^\infty \xi^\lambda e^{-\xi}d\xi=\Gamma \left (\lambda +1,\;\frac{\kappa}{u^2} \right )\;.
\end{equation}
Here by
\begin{equation} \label{Eq_1.6ABC}
\Gamma(\alpha, x)=\int_x^\infty \xi^{\alpha-1}e^{-\xi}d\xi
\end{equation}
is denoted the additional incomplete Gamma-function.

\begin{lemma} \label{L_1} Let  
\begin{equation} \label{Eq_1.8}
F_{\kappa, \lambda, \mu, \nu}(\bs z)=
\int_{{\mathbb S}^{n-1}} \omega_{\kappa, \lambda}\big ((\bs e_\sigma, \bs e)\big )\big |(\bs e_\sigma, \bs e)|^{\mu+\nu}  
(\bs e_\sigma, \bs z) ^{2-\nu} d\sigma\;.
\end{equation}

Then for any $\kappa, \lambda, \mu \geq 0$, $0\leq \nu <2$, the equality
\begin{equation} \label{Eq_1.9}
\max _{|\bs z|=1}F_{\kappa, \lambda, \mu, \nu}(\bs z)=F_{\kappa, \lambda, \mu, \nu}(\bs e)=
\int_{{\mathbb S}^{n-1}} \omega_{\kappa, \lambda}\big ((\bs e_\sigma, \bs e)\big ) 
\big |(\bs e_\sigma, \bs e)|^{\mu +2} d\sigma
\end{equation}
holds.
\end{lemma}
\begin{proof} (i) \textit{The case $\nu=0$}. By (\ref{Eq_1.8}),
\begin{equation} \label{Eq_1.7}
F_{\kappa, \lambda, \mu, 0}(\bs z)=\int_{{\mathbb S}^{n-1}} 
\omega_{\kappa, \lambda}\big ((\bs e_\sigma, \bs e)\big )\big |(\bs e_\sigma, \bs e)|^{\mu}\big |(\bs e_\sigma, \bs z)|^{2} d\sigma\;.
\end{equation}
Let $\bs z'=\bs z- (\bs z, \bs e)\bs e$. 
We choose the Cartesian coordinates with origin ${\mathcal O}$ at the center of the sphere ${\mathbb S}^{n-1}$
such that $\bs e_1=\bs e$ and $\bs e_n$ is collinear to $\bs z'$. Then $\bs z=\alpha \bs e_1 + \beta \bs e_n$, 
where 
\begin{equation} \label{Eq_1.7ABC}
\alpha^2+\beta^2=1.
\end{equation}
 
Now, we rewrite (\ref{Eq_1.7}) in the form
\begin{eqnarray} \label{Eq_1.10}
\hspace{-1cm}& &F_{\kappa, \lambda, \mu, 0}(\bs z)\!=\!\!
\int_{{\mathbb S}^{n-1}} \!\!\omega_{\kappa, \lambda}\big ((\bs e_\sigma, \bs e_1)\big )\big |(\bs e_\sigma, \bs e_1)|^\mu  
\big (\bs e_\sigma, \alpha \bs e_1 + \beta \bs e_n \big ) ^2 d\sigma \nonumber\\
\hspace{-1cm}& &\nonumber\\
\hspace{-1cm}& &=\!\!\int_{{\mathbb S}^{n-1}}\!\!\!\omega_{\kappa, \lambda}\big ((\bs e_\sigma, \bs e_1)\big )\big |(\bs e_\sigma, \bs e_1)|^\mu  
\big [\alpha^2(\bs e_\sigma, \bs e_1)^2\! + \!2\alpha\beta (\bs e_\sigma, \bs e_1)(\bs e_\sigma, \bs e_n)\!+
\!\beta^2(\bs e_\sigma, \bs e_n)^2\big ] d\sigma .
\end{eqnarray}
Let us show that
\begin{equation} \label{Eq_1.11}
\int_{{\mathbb S}^{n-1}}\omega_{\kappa, \lambda}\big ((\bs e_\sigma, \bs e_1)\big ) |(\bs e_\sigma, \bs e_1) |^\mu  
(\bs e_\sigma, \bs e_1)(\bs e_\sigma, \bs e_n) d\sigma=0\;.
\end{equation}
The last equality is obvious for the case $n=2$. We suppose that $n\geq 3$.
We denote by $\vartheta_1, \vartheta_2,\dots,\vartheta_{n-1}$ the spherical coordinates with the center at ${\mathcal O}$, 
where $\vartheta_i \in [0, \pi]$ for $1\leq i\leq n-2$, and $\vartheta_{n-1} \in [0, 2\pi ]$. Then for any 
$\bs \sigma=(\sigma_1,\dots ,\sigma_n)\in {\mathbb S}^{n-1}$ we have
\begin{eqnarray*}
& &\sigma_1=\cos \vartheta_1,\\
& &\sigma_2=\sin \vartheta_1\cos \vartheta_2,\\
& &\dots\dots\dots\dots\dots\dots\dots\dots\\
& &\sigma_{n-1}=\sin \vartheta_1\dots\sin\vartheta_{n-2}\cos \vartheta_{n-1},\\
& &\sigma_n=\sin \vartheta_1\dots\sin\vartheta_{n-2}\sin \vartheta_{n-1}.
\end{eqnarray*}
Using the equalities
$$
(\bs e_\sigma, \bs e_1)=\sigma_1=\cos \vartheta_1,\;\;\;\;\;
(\bs e_\sigma, \bs e_n)=\sigma_n=\sin \vartheta_1\dots \sin\vartheta_{n-2}\sin \vartheta_{n-1}
$$
and
$$
d\sigma=\sin^{n-2}\vartheta_1\sin^{n-3} \vartheta_2\dots\sin\vartheta_{n-2}\;d\vartheta_1 d\vartheta_2 \dots d\vartheta_{n-1},
$$
we calculate the integral on the left-hand side of (\ref{Eq_1.11}):
\begin{eqnarray} \label{Eq_1.12}
\hspace{-1cm}& &\int_{{\mathbb S}^{n-1}}\omega_{\kappa, \lambda}\big ((\bs e_\sigma, \bs e_1)\big )|(\bs e_\sigma, \bs e_1)|^\mu  
(\bs e_\sigma, \bs e_1)(\bs e_\sigma, \bs e_n) d\sigma\nonumber\\
\hspace{-1cm}& &=\!\!\int_0^\pi\!\! ...\!\int_0^\pi\!\int_0^{2\pi}\!\omega_{\kappa, \lambda}\!\big (\cos \vartheta_1 \big )|\cos \vartheta_1|^\mu
\cos \vartheta_1\!\!\left (\prod_{i=1}^{n-2}\sin^{n-i} \vartheta_i \!\right )\! \sin \vartheta_{n-1}
d\vartheta_1 ... d\vartheta_{n-2}d\vartheta_{n-1}\nonumber\\
\hspace{-1cm}& &=I_{\kappa, \lambda}\int_0^\pi\!\! ...\int_0^\pi\
\!\left (\prod_{i=2}^{n-2}\sin^{n-i} \vartheta_i \!\right )d\vartheta_2 ... d\vartheta_{n-2}\!
\int_0^{2\pi}\sin \vartheta_{n-1}d\vartheta_{n-1}\;,
\end{eqnarray}
where
$$
I_{\kappa, \lambda}=\int_0^\pi \omega_{\kappa, \lambda}\!\big (\cos \vartheta_1 \big )|\cos \vartheta_1|^\mu
\cos \vartheta_1 \sin^{n-1} \vartheta_1 d\vartheta_1\;.
$$
Since the inner integral in (\ref{Eq_1.12}) is equal to zero, we arrive at (\ref{Eq_1.11}). 

So, by (\ref{Eq_1.7ABC}), (\ref{Eq_1.10}) and (\ref{Eq_1.11}), we have
\begin{equation} \label{Eq_1.13}
F_{\kappa, \lambda, \mu, 0}(\bs z)\!=\!\!\int_{{\mathbb S}^{n-1}}\!\!\!\omega_{\kappa, \lambda}\big ((\bs e_\sigma, \bs e_1)\big ) 
|(\bs e_\sigma, \bs e_1)|^\mu  
\big [\alpha^2(\bs e_\sigma, \bs e_1)^2\!+\!\beta^2(\bs e_\sigma, \bs e_n)^2\big ] d\sigma\leq \max \{ U,\; V \}\;,
\end{equation}
where
\begin{equation} \label{Eq_1.14}
U=\int_{{\mathbb S}^{n-1}}\!\!\!\omega_{\kappa, \lambda}\big ((\bs e_\sigma, \bs e_1)\big )|(\bs e_\sigma, \bs e_1)|^{\mu+2} d\sigma
\end{equation}
and
\begin{equation} \label{Eq_1.15}
V=\int_{{\mathbb S}^{n-1}}\!\!\!\omega_{\kappa, \lambda}\big ((\bs e_\sigma, \bs e_1)\big )|(\bs e_\sigma, \bs e_1)|^{\mu}(\bs e_\sigma, \bs e_n)^2 d\sigma .
\end{equation}

In view of (\ref{Eq_1.6A}) and the evenness of $\omega_{\kappa, \lambda}(u)$ in $u$, we can write (\ref{Eq_1.14}) as
$$
U=2\omega_{n-1}\int_0^{\pi/2} \omega_{\kappa, \lambda}(\cos \vartheta_1 )\cos ^{\mu+2}\vartheta_1\sin^{n-2}\vartheta_1 d\vartheta_1 .
$$
By the change of variable $\vartheta_1=\frac{\pi}{2}-\varphi$ in the integral on the right-hand side of the last equality, we obtain
\begin{equation} \label{Eq_1.16}
U=2\omega_{n-1}\int_0^{\pi/2} \omega_{\kappa, \lambda}(\sin \varphi )\sin ^{\mu+2} \varphi \cos ^{n-2}\varphi d\varphi\;.
\end{equation}

Now, we calculate the integral on the right-hand side of (\ref{Eq_1.15}):
\begin{eqnarray} \label{Eq_1.17}
\hspace{-1.5cm}& &V=\int_{{\mathbb S}^{n-1}}\omega_{\kappa, \lambda}\big ((\bs e_\sigma, \bs e_1)\big )
|(\bs e_\sigma, \bs e_1)|^{\mu}(\bs e_\sigma, \bs e_n)^2 d\sigma
\nonumber\\
\hspace{-1.5cm}& &\nonumber\\
\hspace{-1.5cm}& &=\int_0^\pi\dots\!\int_0^\pi\!\int_0^{2\pi}\!\omega_{\kappa, \lambda}(\cos \vartheta_1 )|\cos \vartheta_1|^{\mu}
\left ( \prod_{i=1}^{n-1}\sin^{n+1-i}\vartheta_i \right )d\vartheta_1\dots d\vartheta_{n-2}d\vartheta_{n-1}\nonumber\\
\hspace{-1.5cm}& &\nonumber\\
\hspace{-1.5cm}& &=\left \{\int_0^\pi \!\!\omega_{\kappa, \lambda}(\cos \vartheta_1 )|\cos \vartheta_1|^{\mu}\sin^{n}\vartheta_1 d\vartheta_1 \right \}\!
\left \{2\int_0^\pi \!...\int_0^\pi \left ( \prod_{i=2}^{n-1}\sin^{n+1-i}\vartheta_i\! \right ) 
d\vartheta_2...d\vartheta_{n-1}\right \}.
\end{eqnarray}
Putting $\vartheta_1=\varphi+\frac{\pi}{2}$ in the first integral on the right-hand side of (\ref{Eq_1.17}), 
we arrive at equality
\begin{equation} \label{Eq_1.18}
\int_0^\pi \omega_{\kappa, \lambda}(\cos \vartheta_1 )|\cos \vartheta_1|^{\mu}\sin^{n}\vartheta_1 d\vartheta_1
=2\int_0^{\pi/2} \omega_{\kappa, \lambda}(\sin \varphi )\sin ^{\mu} \varphi \cos ^{n}\varphi d\varphi\;.
\end{equation}
Evaluating the multiple integral on the right-hand side of (\ref{Eq_1.17}), we obtain 
\begin{eqnarray*} 
& &2\int_0^\pi ...\int_0^\pi \left ( \prod_{i=2}^{n-1}\sin^{n+1-i}\vartheta_i \right ) 
d\vartheta_2...d\vartheta_{n-1}=2\cdot 2^{n-2}\prod_{k=2}^{n-1}\int_0^{\pi/2}\sin^k\vartheta d\vartheta\\
& &\\
& &=\frac{2^{n-1}}{2^{n-2}}\prod_{k=2}^{n-1} \frac {\Gamma \left (\frac{k+1}{2} \right )\Gamma \left (\frac{1}{2} \right )}
{\Gamma \left (\frac{k+2}{2} \right )}=\frac{2\pi^{(n-1)/2} }{(n-1)\Gamma \left (\frac{n-1}{2} \right )}=\frac{\omega_{n-1}}{n-1}\;,
\end{eqnarray*}
which together with (\ref{Eq_1.17}) and (\ref{Eq_1.18}) leads to
\begin{equation} \label{Eq_1.19}
V=\frac{2\omega_{n-1}}{n-1}\int_0^{\pi/2} \omega_{\kappa, \lambda}(\sin \varphi )\sin ^{\mu} \varphi \cos ^{n}\varphi d\varphi\;.
\end{equation}

Let us show that $U>V$. Integrating by parts in (\ref{Eq_1.16}), we have
\begin{eqnarray*} 
& &\frac{U}{2\omega_{n-1}}=-\frac{1}{n-1}\int_0^{\pi/2}\omega_{\kappa, \lambda}(\sin \varphi )\sin^{\mu+1} \varphi \;d\big(\cos ^{n-1}\varphi \big )\\
& &\\
& &=\frac{1}{n-1}\int_0^{\pi/2}\cos ^{n-1}\varphi\;d\big ( \omega_{\kappa, \lambda}(\sin \varphi )\sin^{\mu+1} \varphi\big )\nonumber\\
& &\\
& &=\frac{1}{n-1}\int_0^{\pi/2}\cos ^{n-1}\varphi \left \{(\mu+1) \sin^{\mu}\varphi \cos \varphi\;\omega_{\kappa, \lambda}(\sin \varphi ) +
\sin^{\mu+1} \varphi \frac{d}{d \varphi} \omega_{\kappa, \lambda}(\sin \varphi ) \right \}d\varphi .
\end{eqnarray*}
In view of (\ref{Eq_1.19}), we can rewrite the last equality as
\begin{equation} \label{Eq_1.20}
\frac{U}{2\omega_{n-1}}=\frac{V}{2\omega_{n-1}}+
\frac{1}{n-1}\int_0^{\pi/2} \left \{\cos ^{n-1}\varphi \sin^{\mu+1} \varphi \frac{d}{d \varphi} \omega_{\kappa, \lambda}(\sin \varphi ) \right \} d\varphi\;.
\end{equation}
By definition (\ref{Eq_1.6AB}) of the function $\omega_{\kappa, \lambda}$, we arrive at
$$
\frac{d}{d \varphi} \omega_{\kappa, \lambda}(\sin \varphi )=\left ( \frac{\kappa}{\sin^2 \varphi}\right )^\lambda 
e^{-\kappa/\sin^2 \varphi}\; \frac{2\kappa\cos \varphi}{\sin^{3} \varphi}> 0\;\;\mbox{for}\;\;\; \varphi \in \left ( 0 , \frac{\pi}{2} \right ),
$$
which together with (\ref{Eq_1.20}) implies
$$
U>V.
$$
This, by (\ref{Eq_1.13}) and (\ref{Eq_1.14}), leads to the inequality
\begin{equation} \label{Eq_1.21}
\max_{|\bs z|=1}F_{\kappa, \lambda, \mu, 0}(\bs z)\leq\!\int_{{\mathbb S}^{n-1}}\!\!\!\omega_{\kappa, \lambda}\big ((\bs e_\sigma, \bs e_1)\big ) 
|(\bs e_\sigma, \bs e_1)|^{\mu+2}d\sigma\;.
\end{equation}
By (\ref{Eq_1.6A}), the value of the integral 
$$
\int_{{\mathbb S}^{n-1}}\!\!\!\omega_{\kappa, \lambda}\big ((\bs e_\sigma, \bs e)\big )|(\bs e_\sigma, \bs e)|^{\mu+2} d\sigma
$$
is independent of $\bs e$. Hence, by (\ref{Eq_1.21}),
\begin{equation} \label{Eq_1.22}
\max_{|\bs z|=1}F_{\kappa, \lambda, \mu, 0}(\bs z)\leq\int_{{\mathbb S}^{n-1}}\!\!\!\omega_{\kappa, \lambda}\big ((\bs e_\sigma, \bs e )\big ) |
(\bs e_\sigma, \bs e )|^{\mu+2}d\sigma\;.
\end{equation}

The obvious lower estimate
$$
\max_{|\bs z|=1}F_{\kappa, \lambda, \mu, 0}(\bs z)\geq F_{\kappa, \lambda, \mu, 0}(\bs e)=
\int_{{\mathbb S}^{n-1}}\omega_{\kappa, \lambda}\big ((\bs e_\sigma, \bs e )\big ) |(\bs e_\sigma, \bs e )|^{\mu+2}d\sigma
$$
together with (\ref{Eq_1.22}), leads to (\ref{Eq_1.9}) for the case $\nu=0$.

\medskip
(ii) \textit{The case $\nu\in (0, 2)$}. By (\ref{Eq_1.6R}), we rewrite (\ref{Eq_1.8}) as
$$
F_{\kappa, \lambda, \mu, \nu}(\bs z)=\int_{{\mathbb S}^{n-1}}\rho_{\kappa, \lambda,\mu}
\big ((\bs e_\sigma, \bs e) \big )f^\nu \big ((\bs e_\sigma, \bs e)\big )f^{2-\nu} \big ((\bs e_\sigma, \bs z)\big )  d\sigma\;.
$$
By part (i) of the proof,
$$
\max_{|\bs z|=1}F_{\kappa, \lambda, \mu, 0}(\bs z)=F_{\kappa, \lambda, \mu, 0}(\bs e)=
\int_{{\mathbb S}^{n-1}}\rho_{\kappa, \lambda,\mu}
\big ((\bs e_\sigma, \bs e) \big )f^{2} \big ((\bs e_\sigma, \bs e)\big )  d\sigma\;,
$$
which, by Corollary \ref{C_1} with $\gamma=2$, implies
$$
\max_{|\bs z|=1}F_{\kappa, \lambda, \mu, \nu}(\bs z)=\int_{{\mathbb S}^{n-1}}\rho_{\kappa, \lambda,\mu}
\big ((\bs e_\sigma, \bs e) \big )f^{2} \big ((\bs e_\sigma, \bs e)\big )  d\sigma\;.
$$
Last inequality combined with (\ref{Eq_1.6R}), proves (\ref{Eq_1.9}) for any $\nu \in (0, 2)$.
\end{proof} 

\section{Weighted estimate for solutions of the Dirichlet problem} \label{S_3}

Here we deal with a solution of the first boundary value problem for the heat equation:
\begin{eqnarray} \label{Eq_4.1}
\left\{\begin{array}{lll}
\displaystyle{\frac{\partial u}{\partial t}=a^2\Delta u}\;,&\quad(x,t)\in {\mathbb R}^n_+\times (0, +\infty),\\
       \\
\displaystyle{u\big |_{t=0}=0}\;,\\
        \\
\displaystyle{u\big |_{x_n=0}=f(x', t)}\;.
\end{array}\right .
\end{eqnarray}
Here $f\in L^p\big ({\mathbb R}^{n-1}\times (0, +\infty ) \big )$, $1\leq p\leq \infty$, and $u$ is 
represented by the heat double layer potential
\begin{equation} \label{Eq_4.2}
u(x, t)=\frac{x_n}{\big ( 4a^2\pi \big )^{n/2}}\int_0^t\int_{{\mathbb R}^{n-1}}\displaystyle{\frac{e^{-\frac{|x-y|^2}{4a^2 (t-\tau)}}}{(t-\tau )^{(n+2)/2}} } f(y', \tau) dy' d\tau
\end{equation}
with $y=(y', 0), y'\in {\mathbb R}^{n-1}$. The norm $\nl f \nr_{p, t}$ was introduced in (\ref{Eq_0.3}).

\begin{proposition} \label{P_4} Let $(x, t)$ be an arbitrary point in ${\mathbb R}^n _+\times (0, +\infty )$. 
The sharp coefficient ${\mathcal W}_p(x, t)$ in the inequality
\begin{equation} \label{Eq_4.3}
\left | \nabla_x \left \{ \frac{u(x, t)}{x_n} \right \}\right |\leq {\mathcal W}_p(x, t) ||f||_{p, t}
\end{equation}
is given by
\begin{equation} \label{Eq_4.4}
{\mathcal W}_p(x,t)=\frac{c_{n,p}}{x_n^{2+\frac{n+1}{p}}}
\max_{|\bs z|=1}\left \{ \int_{{\mathbb S}^{n-1}}\omega_{\kappa , \lambda} \big ( ( \bs e_{\sigma}, \bs e_n )\big )
|( \bs e_{\sigma}, \bs e_n )|^{\frac{n+p+2}{p-1}}|( \bs e_{\sigma}, \bs z )|^{\frac{p}{p-1}} d\sigma\right \}^{\frac{p-1}{p}}\!,
\end{equation}
where
\begin{equation} \label{Eq_4.4A}
c_{n,p}=\frac{2^{\frac{1}{p}}(4a^2)^{1+\frac{1}{p}}}{\pi^{\frac{n}{2}-1}q^{\frac{n}{2}+1+\frac{1}{p}}}\;,
\end{equation} 
$p^{-1}+q^{-1}=1$, $\omega_{\kappa , \lambda}(x)$ is defined by $(\ref{Eq_1.6AB})$ and
\begin{equation} \label{Eq_4.12A}
\kappa=\frac{q x_n^2}{4a^2 t}=\frac{p x_n^2}{4a^2(p-1)t}\;,\;\;\;\;\;\lambda=\frac{(n+4)q}{2}-2=\frac{np+4}{2(p-1)}\;.
\end{equation}

In particular,
\begin{equation} \label{Eq_4.5}
{\mathcal W}_p(x, t)=\frac{c_{n,p}}{x_n^{2+\frac{n+1}{p}}}\left \{ \!2\omega_{n-1}\int_0^{\pi/2} 
\left \{\int_{\frac{q x_n^2}{4a^2 t\cos^2\vartheta}}^\infty \xi^{\frac{np+4}{2(p-1)}}e^{-\xi}d\xi \right \}
 \cos ^{\frac{n+2(p+1)}{p-1}}\vartheta 
\sin^{n-2}\vartheta d\vartheta\!\right \}^{\frac{p-1}{p}}
\end{equation}
for $2\leq p\leq \infty$.

As a special case of $(\ref{Eq_4.5})$ one has
\begin{equation} \label{Eq_4.5A}
{\mathcal W}_\infty(x, t)=\frac{16a^2\sqrt{\pi}}{\Gamma \left (\frac{n-1}{2} \right )x_n^2}\int_0^{\pi/2} 
\left \{\int_{ \frac{x^2_n}{4a^2 t \cos^2\varphi}}^\infty \xi^{n/2} e^{-\xi}d\xi \right \}
\cos ^{2}\vartheta \sin^{n-2}\vartheta d\vartheta\;.
\end{equation}
\end{proposition}
\begin{proof} (i) \textit{General case}. By (\ref{Eq_4.2}), 
$$
\frac{u(x, t)}{x_n}=\frac{1}{\big ( 4a^2\pi \big )^{n/2}}\int_0^t\int_{{\mathbb R}^{n-1}}\displaystyle{\frac{e^{-\frac{|x-y|^2}{4a^2 (t-\tau)}}}{(t-\tau )^{(n+2)/2}} } f(y', \tau) dy' d\tau\;.
$$
Differentiating with respect to $x_j$, $j=1,\dots , n$, we obtain
$$
\nabla_x \left \{ \frac{u(x, t)}{x_n} \right \}=\frac{2\pi}{\big ( 4a^2\pi \big )^{(n+2)/2}}
\int_0^t\int_{{\mathbb R}^{n-1}}\frac{y-x}{(t-\tau )^{(n+4)/2}}e^{-\frac{|x-y|^2}{4a^2 (t-\tau)}}f(y', \tau)dy'd\tau .
$$
Hence,
\begin{equation} \label{Eq_4.6}
\left ( \nabla_x \left \{ \frac{u(x, t)}{x_n} \right \}, \bs z \right )= \frac{2\pi}{\big ( 4a^2\pi \big )^{(n+2)/2}}
\int_0^t\int_{{\mathbb R}^{n-1}}\frac{|y-x|\big ( \bs e_{xy}, \bs z \big )}{(t-\tau )^{(n+4)/2}}e^{-\frac{|x-y|^2}{4a^2 (t-\tau)}}
f(y', \tau)dy'd\tau ,
\end{equation}
where $\bs z$ is a unit $n$-dimensional vector and $\bs e_{xy}=(y-x)/|y-x|$. By (\ref{Eq_4.6}), we conclude that
the sharp coefficient ${\mathcal W}_p(x, t)$ in inequality (\ref{Eq_4.3}) is given by
$$
{\mathcal W}_p(x, t)=\frac{2\pi}{\big ( 4a^2\pi \big )^{n/2}}\max_{|\bs z|=1}\left \{ \int_0^t\int_{{\mathbb R}^{n-1}}
\frac{|y-x|^q |\big ( \bs e_{xy}, \bs z \big ) |^q}{(t-\tau )^{(n+4)q/2}} e^{-\frac{q|y-x|^2}{4a^2 (t-\tau)}}dy' d\tau \right \}^{1/q}. 
$$

We write the last equality in the form
\begin{equation} \label{Eq_4.8}  
{\mathcal W}_p(x, t)=\!\frac{2\pi}{\big ( 4a^2\pi \big )^{n/2}}\max_{|\bs z|=1}\!\left \{\!\!\! \int_{{\mathbb R}^{n-1}}
\!\!\frac{|y\!-\!x|^{q+n}|\big ( \bs e_{xy}, \bs z \big ) |^q}{x_n}\frac{x_n}{|y\!-\!x|^n}  dy'\!\!\int_0^t 
\!\!\frac{e^{-\frac{q|y-x|^2}{4a^2 (t\!-\!\tau)}}}{(t\!-\!\tau )^{(n+4)q/2}} d\tau  \!\right \}^{\!1/q}\!\!\!\!. 
\end{equation}
Setting
$$
s=\frac{q |x-y|^2}{4a^2(t-\tau)}\;,  
$$
we represent the inner integral on the right-hand side of (\ref{Eq_4.8}) as
\begin{equation} \label{Eq_4.9}
\int_0^t \frac{e^{-\frac{q|y-x|^2}{4a^2 (t-\tau)}}}{(t-\tau )^{(n+4)q/2}} d\tau= 
\left (\frac{4a^2}{q|y-x|^2} \right )^{\frac{(n+4)q}{2}-1}\int ^\infty _{\frac{q|y-x|^2}{4a^2 t}} s^{\frac{(n+4)q}{2}-2}e^{-s} ds.
\end{equation}
By (\ref{Eq_1.6AB}) and the equality 
\begin{equation} \label{Eq_4.10}
|y-x||\big ( \bs e_{xy}, \bs e_n \big )|=x_n\; , 
\end{equation}
we write (\ref{Eq_4.9}) as
\begin{equation} \label{Eq_4.11}
\int_0^t \frac{e^{-\frac{q|y-x|^2}{4a^2 (t-\tau)}}}{(t-\tau )^{(n+4)q/2}} d\tau=
\left (\frac{4a^2 \big ( \bs e_{xy}, \bs e_n \big )^2}{qx_n^2} \right )^{\frac{(n+4)q}{2}-1}
\omega_{\kappa , \lambda} \big ( ( \bs e_{xy}, \bs e_n )\big ) ,
\end{equation}
where $\kappa$ and $\lambda$ are defined by (\ref{Eq_4.12A}).

In view of (\ref{Eq_4.10}), we have
\begin{equation} \label{Eq_4.12}
|y-x|^{q+n}=\left ( \frac{x_n}{|\big ( \bs e_{xy}, \bs e_n \big )|} \right )^{q+n}\;,
\end{equation}
which, in combination with (\ref{Eq_4.8}) and (\ref{Eq_4.11}), leads to
\begin{equation} \label{Eq_4.13}
{\mathcal W}_p(x, t)=\!\frac{2(4a^2)^{1+\frac{1}{p}}}{\pi^{\frac{n}{2}\!-\!1}q^{\frac{n}{2}\!+\!1+\!\frac{1}{p}}x_n^{2\!+\!\frac{n\!+\!1}{p}}}
\max_{|\bs z|=1}\!\left \{\! \int_{{\mathbb S}^{n-1}_-}\!\!\!\omega_{\kappa , \lambda} \big ( ( \bs e_{\sigma}, \bs e_n )\big )
|( \bs e_{\sigma}, \bs e_n )|^{\frac{n\!+\!p\!+\!2}{p-1}}|( \bs e_{\sigma}, \bs z )|^{\frac{p}{p\!-\!1}} d\sigma\right \}^{\frac{p\!-\!1}{p}}\!\!\!, 
\end{equation}
where 
${\mathbb S}^{n-1}_- =\{ \sigma \in {\mathbb S}^{n-1}: (\bs e_\sigma , \bs e_n )<0 \}$.

Using the evenness of the integrand in (\ref{Eq_4.13}) with respect to $\bs e_\sigma$, we obtain
\begin{equation} \label{Eq_4.14}
{\mathcal W}_p(x,t)=\!\frac{2^{\frac{1}{p}}(4a^2)^{1+\frac{1}{p}}}{\pi^{\frac{n}{2}\!-\!1}q^{\frac{n}{2}\!+\!1+
\!\frac{1}{p}}x_n^{2\!+\!\frac{n\!+\!1}{p}}}
\max_{|\bs z|=1}\!\left \{\! \int_{{\mathbb S}^{n-1}}\!\!\!\omega_{\kappa , \lambda } \big ( ( \bs e_{\sigma}, \bs e_n )\big )
|( \bs e_{\sigma}, \bs e_n )|^{\frac{n\!+\!p\!+\!2}{p\!-\!1}}|( \bs e_{\sigma}, \bs z )|^{\frac{p}{p\!-\!1}} d\sigma\right \}^{\frac{p\!-\!1}{p}}\!\!\!, 
\end{equation} 
which proves (\ref{Eq_4.4}).

(ii) \textit{The case $p\in [2, \infty]$}. Solving the system
$$
2-\nu=\frac{p}{p-1}\;,\;\;\;\;\;\;\mu+\nu=\frac{n+p+2}{p-1}
$$
with respect to $\nu$ and $\mu$, we arrive at 
$$
\nu=\frac{p-2}{p-1}\;,\;\;\;\;\;\;\mu=\frac{n+4}{p-1}\;.
$$
So, $\mu>0$ for any $p>1 $ and $\nu \in [0, 1)$ for $p\geq 2$. Applying Lemma \ref{L_1} to (\ref{Eq_4.14}), we conclude
\begin{equation} \label{Eq_4.15}
{\mathcal W}_p(x,t)=\frac{c_{n,p}}{x_n^{2+\frac{n+1}{p}}}
\left \{\int_{{\mathbb S}^{n-1}}\omega_{\kappa , \lambda } \big ( ( \bs e_{\sigma}, \bs e_n )\big )
|( \bs e_{\sigma}, \bs e_n )|^{\frac{n+2(p+1)}{p-1}} d\sigma\right \}^{\frac{p-1}{p}} , 
\end{equation}
where $p\in [2, \infty]$ and the constant $c_{n, p}$ is defined by (\ref{Eq_4.4A}). By (\ref{Eq_1.6A}) and
(\ref{Eq_1.6AB}), we write (\ref{Eq_4.15}) as (\ref{Eq_4.5}). 
\end{proof}

\section{Estimate for solutions of the Neumann problem} \label{S_4}

Let us consider the Neumann problem for the heat equation:
\begin{eqnarray} \label{Eq_6.1}
\left\{\begin{array}{lll}
\displaystyle{\frac{\partial u}{\partial t}=a^2\Delta u}\;,&\quad(x,t)\in {\mathbb R}^n_+\times (0, +\infty),\\
       \\
\displaystyle{u\big |_{t=0}=0}\;,\\
        \\
\displaystyle{\frac{\partial u}{\partial x_n}\Big |_{x_n=0}=g(x', t)}
\end{array}\right .
\end{eqnarray}
with $g\in L^p\big ({\mathbb R}^{n-1}\times (0, +\infty ) \big )$, $1\leq p \leq \infty$. 
Here $u$ is represented as the heat single layer potential
\begin{equation} \label{Eq_6.2}
u(x, t)=-\frac{2a^2}{\big ( 4a^2\pi \big )^{n/2}}\int_0^t\int_{{\mathbb R}^{n-1}}\displaystyle{\frac{e^{-\frac{|x-y|^2}{4a^2 (t-\tau)}}}{(t-\tau )^{n/2}} }g(y', \tau) dy' d\tau\;,
\end{equation}
where $y=(y', 0), y'\in {\mathbb R}^{n-1}$.

\begin{proposition} \label{P_5} Let $(x, t)$ be an arbitrary point in ${\mathbb R}^n _+\times (0, +\infty )$.
The sharp coefficient ${\mathcal N}_p(x, t)$ in the inequality
\begin{equation} \label{Eq_6.3}
|\nabla_x u(x, t)|\leq {\mathcal N}_p(x, t) ||g||_{p, t}
\end{equation}
is given by
\begin{equation} \label{Eq_6.4}
{\mathcal N}_p(x, t)=\frac{k_{n,p}}{x_n^{\frac{n+1}{p}}}\max_{|\bs z|=1}
\left \{ \int_{{\mathbb S}^{n-1}}\omega_{\kappa , \lambda} \big ( ( \bs e_{\sigma}, \bs e_n )\big )
|( \bs e_{\sigma}, \bs e_n )|^{\frac{n-p+2}{p-1}}|( \bs e_{\sigma}, \bs z )|^{\frac{p}{p\!-\!1}} d\sigma\right \}^{\frac{p-1}{p}}\!,
\end{equation}
where
\begin{equation} \label{Eq_6.4A}
k_{n,p}=\frac{2^{(3-p)/p} a^{2/p}}{\pi^{n/2}q^{\frac{n}{2}+\frac{1}{p}}}\;,
\end{equation}
$p^{-1}+q^{-1}=1$, $\omega_{\kappa , \lambda}(x)$ is defined by $(\ref{Eq_1.6AB})$ and
\begin{equation} \label{Eq_6.12A}
\kappa=\frac{q x_n^2}{4a^2 t}=\frac{p x_n^2}{4a^2(p-1)t}\;,\;\;\;\;\;\lambda=\frac{(n+2)q}{2}-2=\frac{(n-2)p+4}{2(p-1)}\;.
\end{equation}

In particular,
\begin{equation} \label{Eq_6.5}
{\mathcal N}_p(x, t)=\frac{k_{n,p}}{x_n^{\frac{n+1}{p}}}
\left \{ 2\omega_{n-1}\!\! \int_0^{\pi/2} \!\!
\left \{\int_{\frac{q x_n^2}{4a^2 t \cos^2\vartheta}}^\infty \xi^{\frac{(n-2)p+4}{2(p-1)}}e^{-\xi}d\xi \right \}
 \cos ^{\frac{n+2}{p-1}}\vartheta \sin^{n-2}\vartheta d\vartheta\right \}^{\frac{p\!-\!1}{p}}
\end{equation}
for $2\leq p\leq (n+4)/2$.

As a special case of $(\ref{Eq_6.5})$ one has
\begin{equation} \label{Eq_6.5A}
{\mathcal N}_2(x, t)=\frac{b_{n}}{x_n^{\frac{n+1}{2}}}\left \{\int_0^{\pi/2} 
\left \{\int_{ \frac{x^2_n}{2a^2 t \cos^2\vartheta }}^\infty \xi^{n} e^{-\xi}d\xi \right \}
\cos ^{n+2}\vartheta \cos^{n-2}\vartheta d\vartheta \right \}^{1/2}\;,
\end{equation}
where
$$
b_{n}=\frac{a}{2^{\frac{n-1}{2}}\pi^{\frac{n+1}{4}}\sqrt{\Gamma \left (\frac{n-1}{2} \right )}}\;.
$$
\end{proposition}
\begin{proof} (i) \textit{General case}. Differentiating in (\ref{Eq_6.2}) with respect to $x_j$, $j=1,\dots , n$, we obtain
$$
\nabla_x u(x, t)=-\frac{1}{\big ( 4a^2\pi \big )^{n/2}}\int_0^t\int_{{\mathbb R}^{n-1}}\frac{y-x}{(t-\tau )^{(n+2)/2}} 
e^{-\frac{|x-y|^2}{4a^2 (t-\tau)}}g(y', \tau)dy'd\tau ,
$$
which leads to
\begin{equation} \label{Eq_6.6}
\big ( \nabla_x u(x, t), \bs z \big )= -\frac{1}{\big ( 4a^2\pi \big )^{n/2}}\int_0^t\int_{{\mathbb R}^{n-1}}
\frac{|y-x|\big ( \bs e_{xy}, \bs z \big )}{(t-\tau )^{(n+2)/2}}e^{-\frac{|x-y|^2}{4a^2 (t-\tau)}}g(y', \tau)dy'd\tau ,
\end{equation}
where $\bs z$ is a unit $n$-dimensional vector and $\bs e_{xy}=(y-x)/|y-x|$. It follows from (\ref{Eq_6.6}) that
the sharp coefficient ${\mathcal N}_p(x, t)$ in inequality (\ref{Eq_6.3}) is given by
$$
{\mathcal N}_p(x, t)=\frac{1}{\big ( 4a^2\pi \big )^{n/2}}\max_{|\bs z|=1}\left \{ \int_0^t\int_{{\mathbb R}^{n-1}}
\frac{|y-x|^q |\big ( \bs e_{xy}, \bs z \big ) |^q}{(t-\tau )^{(n+2)q/2}} e^{-\frac{q|y-x|^2}{4a^2 (t-\tau)}}dy' d\tau \right \}^{1/q}. 
$$
Now, we write the last equality as
\begin{equation} \label{Eq_6.8}
{\mathcal N}_p(x, t)=\!\frac{1}{\big ( 4a^2\pi \big )^{n/2}}\max_{|\bs z|=1}\!\left \{\! 
\int_{{\mathbb R}^{n-1}}\!\!\!\frac{|y\!-\!x|^{q+n}|\big ( \bs e_{xy}, \bs z \big ) |^q}{x_n}\frac{x_n}{|y\!-\!x|^n}  dy'\!\!\!
\int_0^t \!\!\frac{e^{-\frac{q|y-x|^2}{4a^2 (t\!-\!\tau)}}}{(t\!-\!\tau )^{(n+2)q/2}} d\tau \!\right \}^{1/q}\!\!\!\!. 
\end{equation}
Putting
$$
s=\frac{q |x-y|^2}{4a^2(t-\tau)}\;,  
$$
we represent the inner integral on the right-hand side of (\ref{Eq_6.8}) in the form
\begin{equation} \label{Eq_6.9}
\int_0^t \frac{e^{-\frac{q|y-x|^2}{4a^2 (t-\tau)}}}{(t-\tau )^{(n+2)q/2}} d\tau= 
\left (\frac{4a^2}{q|y-x|^2} \right )^{\frac{(n+2)q}{2}-1}\int ^\infty _{\frac{q|y-x|^2}{4a^2 t}} s^{\frac{(n+2)q}{2}-2}e^{-s} ds.
\end{equation}
By (\ref{Eq_1.6AB}) and equality (\ref{Eq_4.10}), 
we write (\ref{Eq_6.9}) as follows
\begin{equation} \label{Eq_6.11}
\int_0^t \frac{e^{-\frac{q|y-x|^2}{4a^2 (t-\tau)}}}{(t-\tau )^{(n+2)q/2}} d\tau=
\left (\frac{4a^2 \big ( \bs e_{xy}, \bs e_n \big )^2}{qx_n^2} \right )^{\frac{(n+2)q}{2}-1}
\omega_{\kappa , \lambda } \big ( ( \bs e_{xy}, \bs e_n )\big ) ,
\end{equation}
where $\kappa$ and $\lambda $ are defined by (\ref{Eq_6.12A}).

Using (\ref{Eq_4.12}) and substituting (\ref{Eq_6.11}) into (\ref{Eq_6.8}), we arrive at the representation
\begin{equation} \label{Eq_6.13}
{\mathcal N}_p(x, t)=\!\frac{(4a^2)^{1/p}}{\pi^{n/2}q^{\frac{n}{2}\!+\!\frac{1}{p}}x_n^{\frac{n+1}{p}}}\max_{|\bs z|=1}\!
\left \{\! \int_{{\mathbb S}^{n-1}_-}\!\!\!\omega_{\kappa , \lambda } \big ( ( \bs e_{\sigma}, \bs e_n )\big )
|( \bs e_{\sigma}, \bs e_n )|^{\frac{n-p+2}{p-1}}|( \bs e_{\sigma}, \bs z )|^{\frac{p}{p\!-\!1}} d\sigma\right \}^{\frac{p\!-\!1}{p}}\!\!\!, 
\end{equation}
where ${\mathbb S}^{n-1}_- =\{ \sigma \in {\mathbb S}^{n-1}: (\bs e_\sigma , \bs e_n )<0 \}$.

In view of the evenness of the integrand in (\ref{Eq_6.13}) with respect to $\bs e_\sigma$, we obtain
\begin{equation} \label{Eq_6.14}
{\mathcal N}_p(x, t)=\!\frac{2^{(3-p)/p} a^{2/p}}{\pi^{n/2}q^{\frac{n}{2}\!+\!\frac{1}{p}}x_n^{\frac{n+1}{p}}}\max_{|\bs z|=1}\!
\left \{\! \int_{{\mathbb S}^{n-1}}\!\!\!\omega_{\kappa , \lambda} \big ( ( \bs e_{\sigma}, \bs e_n )\big )
|( \bs e_{\sigma}, \bs e_n )|^{\frac{n-p+2}{p-1}}|( \bs e_{\sigma}, \bs z )|^{\frac{p}{p\!-\!1}} d\sigma\right \}^{\frac{p\!-\!1}{p}}\!\!\!,
\end{equation}
which proves (\ref{Eq_6.4}).

(ii) \textit{The case $p\in [2, (n+4)/2]$}. Solving the system
$$
2-\nu=\frac{p}{p-1}\;,\;\;\;\;\;\;\mu+\nu=\frac{n-p+2}{p-1}
$$
with respect to $\nu$ and $\mu$, we obtain
$$
\nu=\frac{p-2}{p-1}\;,\;\;\;\;\;\;\mu=\frac{n-2p+4}{p-1}\;.
$$
Therefore, the conditions $0\leq \nu <2, \mu\geq 0$ hold for $p\in [2, (n+4)/2]$. Applying Lemma \ref{L_1} to
(\ref{Eq_6.14}), we arrive at
\begin{equation} \label{Eq_6.15}
{\mathcal N}_p(x,t)=\frac{k_{n,p}}{x_n^{\frac{n+1}{p}}}
\left \{\int_{{\mathbb S}^{n-1}}\omega_{\kappa , \lambda } \big ( ( \bs e_{\sigma}, \bs e_n )\big )
|( \bs e_{\sigma}, \bs e_n )|^{\frac{n+2}{p-1}} d\sigma\right \}^{\frac{p-1}{p}} , 
\end{equation}
where $p\in [2, (n+4)/2]$ and the constant $k_{n, p}$ is defined by (\ref{Eq_6.4A}). In view of (\ref{Eq_1.6A}) and
(\ref{Eq_1.6AB}), we write (\ref{Eq_6.15}) as (\ref{Eq_6.5}).
\end{proof}


\end{document}